\documentclass{amsart}

\newtheorem{theorem}{Theorem}[section]
\newtheorem{lem}[theorem]{Lemma}
\newtheorem{prop}[theorem]{Proposition}
\theoremstyle{definition}

\theoremstyle{remark}

\DeclareRobustCommand{\stirling}{\genfrac\{\}{0pt}{}}
\numberwithin{equation}{section}

\begin{document}

\title{A $q$-analogue for Euler's evaluations of the Riemann zeta function}
%A $q$ - analogue for Euler's $\zeta(2k)=\dfrac{(-1)^{k+1}2^{2k}B_{2k}\pi^{2k}}{2(2k)!}$
%    Remove any unused author tags.

%    author one information
\author{Ankush Goswami}
\address{Department of Mathematics, University of Florida, 
Gainesville, Fl 32603}
%\curraddr{}
\email{ankush04@ufl.edu}
\thanks{}

%    author two information
%\author{}
%\address{}
%\curraddr{}
%\email{}
%\thanks{}



\date{}

%\dedicatory{}

\begin{abstract}
We provide a $q$-analogue of Euler's formula for $\zeta(2k)$ for $k\in\mathbb{Z}^+$. Our main results are stated in Theorems 3.1 and 3.2 below. The result generalizes a recent result of Z.W. Sun who obtained $q$-analogues of $\zeta(2)=\pi^2/6$ and $\zeta(4)=\pi^4/90$.  
%In the process we also obtain some interesting properties of binomial coefficients and stirling numbers of second kind.   
\end{abstract}

\maketitle
\section{Introduction}\mbox{}\\
Recently, Sun obtained a very nice $q$-analogue of Euler's formula $\zeta(2)=\pi^2/6$. Motivated by this, the author obtained the $q$-analogue of $\zeta(4)$ and noted that it was simultaneously and independently obtained by Sun \cite{Sun}. The author then obtained the $q$-analogue of $\zeta(6)$ in \cite{AG} but realized that this transition to the $q$-analogue of $\zeta(6)$ is more difficult as compared to $\zeta(2)$ and $\zeta(4)$. This difficulty arises due to an extra term that shows up in the identity; however in the limit as $q\rightarrow 1^{-}$, this term $\rightarrow 0$. Thus it is necessary to study the $q$-analogue of Euler's celebrated formula
\begin{eqnarray}\label{Euler}
\zeta(2k)=\dfrac{(-1)^{k+1}2^{2k}B_{2k}\pi^{2k}}{2(2k)!}
\end{eqnarray}  
for all $k\in\mathbb{Z}^+$. We will see shortly that this requires a consideration of two cases: $k$ even and $k$ odd separately (see Theorems \ref{main1}, \ref{main2} below). We also mention here that Zudilin \cite{Z} and Krattenthaler-Rivoal-Zudilin \cite{KRZ} have studied the Diophantine properties of $q$-zeta values, the sums appearing in the left-hand side of Theorems \ref{main1} and \ref{main2}.  
%For the general $q$-analogue formulation of Euler's formula for $\zeta(2k) \;(k=1,2,3...)$, we need to deal with two cases separately. We will see that the case when $k$ is an even number is easier to deal with than the case when $k$ is an odd number. 
\section{Notations}
For a positive integer $k$ we use the following standard notations. Let $B_{2k}$ denote the $2k$th Bernoulli number. Let $\stirling{n}{k}$ denote a Stirling number of the second kind, which is the number of ways of partitioning a set of $n$ objects into $k$ non-empty subsets. Let the complex upper half-plane be denoted by $\mathcal{H}=\{\tau\in\mathbb{C} :\mbox{Im}(\tau)>0\}$ and let  $SL_2(\mathbb{Z})$ denote the full modular group which is defined to be the set of all $2\times 2$ matrices with integer entries and determinant one. Also let $\Gamma_0(4)$ denote the well-known principal congruence subgroup of $SL_2(\mathbb{Z})$ defined by
\begin{eqnarray*}
\Gamma_0(4)=\left\{\left(\begin{array}{cc}a & b\\c & d\end{array}\right) \in SL_2(\mathbb{Z}): c\equiv 0\;(\mbox{mod 4})\right\},
\end{eqnarray*}
Finally, let $\mathcal{M}_{2k}(\Gamma_0(4))$ be the vector space of all weight $2k$ modular forms over $\Gamma_0(4)$ and $\mathcal{S}_{2k}(\Gamma_0(4))$ denote the subspace of $\mathcal{M}_{2k}(\Gamma_0(4))$ of all weight $2k$ cusp forms over $\Gamma_0(4)$. Let $q=e^{2\pi i\tau}$ where $\tau\in\mathcal{H}$. We define the Dedekind eta function, a well-known modular form of weight 1/2, by 
\begin{eqnarray}
\eta(\tau) = q^{1/24}\displaystyle\prod_{n=1}^\infty (1-q^n).
\end{eqnarray}
For a thorough treatment on modular forms, the interested readers should consult \cite{Koblitz, DS}. Finally, we denote the $n$th triangular number $T_n$ by
\begin{eqnarray}
T_n=\dfrac{n(n+1)}{2}, \;\;n=1,2,3....
\end{eqnarray}
and the corresponding generating function by 
\begin{eqnarray}
\psi(q)=\displaystyle\sum_{n=1}^\infty q^{T_n}.  
\end{eqnarray}
Also let $d_k$ be given by
\begin{eqnarray*}
d_k=-\dfrac{(-16)^k B_{2k}(4^k-1)}{8k}\in\mathbb{Q}.
\end{eqnarray*}
\section{Main theorems}
\begin{theorem}\label{main1}
Let $k\geq 2$ be an even integer. For a complex number $q$ with $|q|<1$ we have
\begin{eqnarray}\label{qaz4}
\ \ \ \ \ \ \ \ \ \ \ \ \sum_{n=0}^\infty \dfrac{2^{2k-1}q^{2n+1}\;P^e_{2k-2}(q^{2n+1})}{(1-q^{2n+1})^{2k}}-T_{2k}(\tau/2)=q^{k/2}d_k\prod_{n=1}^\infty\dfrac{(1-q^{2n})^{4k}}{(1-q^{2n-1})^{4k}}
\end{eqnarray}
where 
\begin{eqnarray}\label{polyeven}
P^e_{2k-2}(z)=\displaystyle\sum_{l=1}^{2k-1}(-1)^lb_{k}(l)z^{l-1}
\end{eqnarray} 
is a polynomial of degree $(2k-2)$ with integer coefficients
\begin{eqnarray}\label{evenbs}
b_k(l)=\sum_{m=0}^{2k-1}(-1)^m a_{k}(m)\binom{2k-m-1}{l}
\end{eqnarray}
where the $a_k(m)$ are defined by 
\begin{eqnarray}\label{evenas}
a_k(m)=\sum_{j=0}^{2k-1}j!(-1)^j\stirling{2k-1}{j}\binom{j}{m}.
\end{eqnarray}
%Also, $d_k$ is given by 
%\begin{eqnarray}
%d_k=-\dfrac{(-16)^k B_{2k}(4^k-1)}{8k}\in\mathbb{Q}
%\end{eqnarray}
and $T_{2k}(\tau)\in\mathcal{S}_{2k}(\Gamma_0(4))$, thus $T_{2k}(\tau/2)\rightarrow 0$ as $q\rightarrow 1$, where the limit is taken from inside the unit disk. In other words, Theorem $\ref{main1}$ gives a $q$-analogue of $(\ref{Euler})$ for $\zeta(2k)$ with $k$ even.
\end{theorem}
\begin{theorem}\label{main2}
Let $k\geq 1$ be an odd integer. For a complex number $q$ with $|q|<1$ we have
\begin{eqnarray}\label{qaz5}
\sum_{n=0}^\infty \dfrac{q^{2n+1}\;P^o_{4k-2}(q^{2n+1})}{(1-q^{2(2n+1)})^{2k}}-T_{2k}(\tau)=q^{k}d_k\prod_{n=1}^\infty\dfrac{(1-q^{4n})^{4k}}{(1-q^{4n-2})^{4k}}
\end{eqnarray}
where 
\begin{eqnarray}\label{polyodd}
P^o_{4k-2}(z)=(1+z)^{2k}P^e_{2k-2}(z)-2^{2k-1}zP^e_{2k-2}(z^2)
\end{eqnarray}
is a polynomial of degree $(4k-2)$ with integer coefficients and where $P^e_{2k-2}(z)$ is the polynomial defined in $(\ref{polyeven})$
% Also, $d_k$ is given by 
%\begin{eqnarray}
%d_k=-\dfrac{(-16)^k B_{2k}(4^k-1)}{8k}\in\mathbb{Q}
%\end{eqnarray}
and $T_{2k}(\tau)\in\mathcal{S}_{2k}(\Gamma_0(4))$, thus $T_{2k}(\tau)\rightarrow 0$ as $q\rightarrow 1$, where the limit is taken from inside the unit disk. In other words, Theorem $\ref{main2}$ gives a $q$-analogue of $(\ref{Euler})$ for $\zeta(2k)$ with $k$ odd.
\end{theorem}\mbox{}\\
\textit{Two remarks:} 
\begin{enumerate}
\item Note that we are using (\ref{psi}) of Theorem \ref{Ono} (see below) to prove Theorems \ref{main1} and \ref{main2}. Clearly the left-hand of (\ref{psi}) is $q^k$ times a function of $q^2$. If the right-hand side of (\ref{psi}) also turns out to be a function of $q^2$, we can replace $q\rightarrow\sqrt{q}$ without affecting our results. This happens to be the case in Theorem \ref{main1} where we obtain expressions in $q^2$ for both the sum and product, thereby giving us $T_{2k}(\tau/2)$ in (\ref{qaz4}). However we do not obtain such expressions in $q^2$ on both sides of (\ref{qaz5}) in Theorem \ref{main2} and thus we get $T_{2k}(\tau)$ instead of $T_{2k}(\tau/2)$. However numerical calculations for $k=1, 3, 5$ suggest that we are likely to get expressions involving $q^2$ for the sum in (\ref{qaz5}) so that we can replace $q\rightarrow\sqrt{q}$, thereby getting $T_{2k}(\tau/2)$ in (\ref{qaz5}).
\item The cusp form $T_{2k}(\tau)$ in Theorems \ref{main1}, \ref{main2} is well-defined and uniquely determined by the difference of a $q$-series and a $q$-product as follows:
$$\ \ \ \ \ \ \ \ \ \ \ \ T_{2k}(\tau)=\left\{\begin{array}{cc}
\displaystyle\sum_{n=0}^\infty \dfrac{2^{2k-1}q^{4n+2}\;P^e_{2k-2}(q^{4n+2})}{(1-q^{4n+2})^{2k}}-q^{k}d_k\displaystyle\prod_{n=1}^\infty\dfrac{(1-q^{4n})^{4k}}{(1-q^{4n-2})^{4k}} \;\; (k\;\mbox{even})\\
 \\
\displaystyle\sum_{n=0}^\infty \dfrac{q^{2n+1}\;P^o_{4k-2}(q^{2n+1})}{(1-q^{4n+2})^{2k}}-q^{k}d_k\displaystyle\prod_{n=1}^\infty\dfrac{(1-q^{4n})^{4k}}{(1-q^{4n-2})^{4k}}\;\;\;\;\;\;\;\;\; (k\;\mbox{odd})
\end{array}\right.$$
\end{enumerate}    
%\begin{theorem}
%For a complex $q$ with $|q|<1$ we have
%\begin{eqnarray}\label{qaz6}
%\sum_{k=0}^\infty \dfrac{q^{k}(1+q^{2k+1})\;P_4(q^{2k+1})}{(1-q^{2k+1})^6}-\phi^{12}(q)=256q\prod_{n=1}^\infty\dfrac{(1-q^{2n})^{12}}{(1-q^{2n-1})^{12}}
%\end{eqnarray}
%where $P_4(x) = x^4+236x^3+1446x^2+236x+1$ and $\phi(q)=\displaystyle\prod_{n=1}^\infty (1-q^n)$ is the Euler's function. In other words, $(\ref{qaz6})$ gives a $q$-analogue of $\zeta(6)=\pi^6/945$.
%\end{theorem}
\section{Some useful lemmas}
We next state an important theorem which follows from Jacobi triple product identity, originally proved by Gauss (see \cite{B}, p.10, Cor. 1.3.4 and notes in p.23).
\begin{lem}\label{Gauss}
For $|q|<1$ we have
\begin{eqnarray}
\psi(q)=\prod_{n=1}^\infty \dfrac{(1-q^{2n})}{(1-q^{2n-1})}.
\end{eqnarray}
\end{lem}\mbox{}\\
Thus Lemma $\ref{Gauss}$ yields 
\begin{eqnarray}\label{Gausscor}
%&&\prod_{n=1}^\infty\dfrac{(1-q^{2n})^8}{(1-q^{2n-1})^8} = \psi^8(q) = \sum_{n=1}^\infty t_8(n) q^n\\
%&\mbox{and,}&\nonumber\\
\psi^{4k}(q) = \prod_{n=1}^\infty\dfrac{(1-q^{2n})^{4k}}{(1-q^{2n-1})^{4k}} = \sum_{n=1}^\infty t_{4k}(n) q^n
\end{eqnarray} 
where $t_{4k}(n)$ is the number of ways of representing a positive integer $n$ as a sum of $4k$ triangular numbers. \\
Next, the following well-known result in \cite{ARBILEO} due to Atanosov et al. gives us an exact formula for $t_{4k}(n)$. Indeed, the authors show that $t_{4k}(n)$ behaves when $n$ becomes large like $\sigma^{\#}_{2k-1}(2n+k)$, the modified divisor function defined by
\begin{eqnarray}\label{ModD}
\sigma^{\#}_{k}(n):=\sum_{\substack{d|n\\n/d\;odd}}d^k =\left\{\begin{array}{cc}\sigma_{k}(n)\;\;\;\;\;\;\;\;\;\mbox{when $n$ is odd},\\2^{k}\sigma^{\#}_{k}\left(\dfrac{n}{2}\right)\;\mbox{when $n$ is even}\end{array}\right.
\end{eqnarray} 
where $\sigma_k(n)$ is the $k$th divisor function defined as
\begin{eqnarray}
\sigma_{k}(n) = \sum_{d|n}d^k.
\end{eqnarray} 
\begin{theorem}\label{Ono}
Let $k\in\mathbb{N}$. Then 
\begin{eqnarray}\label{psi}
q^k\psi^{4k}(q^2)=\dfrac{1}{d_k}(H_{2k}(\tau)-T_{2k}(\tau))
\end{eqnarray}
where $d_k$ is defined as in Theorem $\ref{main1}$ and $\ref{main2}$, $T_{2k}(\tau)\in\mathcal{S}(\Gamma_0(4))$ and $H_{2k}(\tau)$ is an Eisenstein series of weight $2k$ on $\Gamma_0(4)$ defined by
\begin{eqnarray}\label{Eis}
H_{2k}(\tau) = \left\{\begin{array}{cc}\displaystyle\sum_{\substack{n>0 \\ n \;even}}\sigma^{\#}_{2k-1}(n)q^n\;\mbox{for $k$ even},\\ \displaystyle\sum_{\substack{n>0\\n \;odd}}\sigma^{\#}_{2k-1}(n)q^n\;\mbox{for $k$ odd}.\end{array}\right.
\end{eqnarray}
\end{theorem}\mbox{}\\
By comparing cofficients in $(\ref{psi})$, Atanosov et al. obtain the following expression for $t_{4k}(n)$ in \cite{ARBILEO} (Cor. 2.6, p.119):
\begin{eqnarray}
t_{4k}(n) = \dfrac{1}{d_k}(\sigma^{\#}_{2k-1}(2n+k)-a(2n+k))
\end{eqnarray}
where $T_{2k}(\tau)=\displaystyle\sum_{n=0}^\infty a(n)q^n\in\mathcal{S}_{2k}(\Gamma_0(4))$. Indeed, Theorem \ref{Ono} follows easily from \cite{ORW} where the authors detail a closed formula for $t_{4k}(n)$.\\We next state an important theorem for the generating function transformation involving Stirling numbers. Let $F(z)$ denote the infinite geometric series  
\begin{eqnarray}
F(z):=\sum_{n=0}^\infty z^n = \dfrac{1}{1-z}
\end{eqnarray} 
with $|z|<1$. Then choosing $f_n=1$ in Prop. 3.1, p.135 of \cite{Smidt} we obtain
\begin{prop}\label{Stir}
Let $l$ be a fixed positive integer. Then we have
\begin{eqnarray}\label{genS}
\sum_{n=0}^\infty n^l z^n = \sum_{j=0}^l j!\stirling{l}{j}\dfrac{z^j}{(1-z)^{j+1}}.
\end{eqnarray} 
\end{prop}
The proof follows by induction on $l$ in conjuction with the recurrence relation of Stirling numbers
\begin{eqnarray}\label{rec}
\stirling{n}{l} = l\stirling{n-1}{l}+\stirling{n-1}{l-1}.
\end{eqnarray}
\section{Proofs of Theorems \ref{main1} and \ref{main2}}  
%\begin{theorem}[Ono, Robins, Wahl]
%For a positive integer $n$, we have
%\begin{eqnarray}
%t_8(n) = \sigma^\#_3(n+1)
%\end{eqnarray}
%where
%\begin{eqnarray}
%\sigma^\#_3(n) = \sum_{\substack{d|n \\ n/d \;\equiv 1\;(mod\;2)}}d^3
%\end{eqnarray} 
%\end{theorem}
%\begin{theorem}[Ono, Robins, Wahl]
%Let, $\eta^{12}(2\tau)=\displaystyle\sum_{k=0}^\infty a(2k+1)q^{2k+1}$ then for a positive integer $n$ we have
%\begin{eqnarray}
%t_{12}(n) = \dfrac{\sigma_5(2n+3)-a(2n+3)}{256}
%\end{eqnarray}
%where
%\begin{eqnarray}
%\sigma_5(n) = \sum_{d|n}d^5
%\end{eqnarray} 
%\end{theorem}
%\section{Proof of Theorem 2.2}
%Since $\zeta(6)=\dfrac{\pi^6}{945}$ has the following equivalent form
%%\begin{eqnarray}
%%\sum_{k=0}^\infty \dfrac{1}{(2k+1)^4} = \dfrac{15}{16}\zeta(4) = \dfrac{\pi^4}{96}
%%\end{eqnarray}
%%and,
Since $\zeta(2k)=\dfrac{(-1)^{k+1}2^{2k}B_{2k}}{2(2k)!}\pi^{2k}$ has the following equivalent form
\begin{eqnarray}
\sum_{n=0}^\infty \dfrac{1}{(2n+1)^{2k}} = \left(\dfrac{2^{2k}-1}{2^{2k}}\right)\zeta(2k) = \dfrac{(-1)^{k+1}(4^k-1)B_{2k}}{2(2k)!}\pi^{2k}
\end{eqnarray}
it will be sufficient to get the $q$-analogue of (5.1). From the $q$-analogue of Euler's Gamma function we know that 
\begin{eqnarray}\label{EG}
\lim_{q\uparrow 1} \;(1-q)\prod_{n=1}^\infty\dfrac{(1-q^{2n})^2}{(1-q^{2n-1})^2}=\dfrac{\pi}{2}
\end{eqnarray}
so that from $(\ref{EG})$ we have
%\begin{eqnarray}
%\lim_{q\uparrow 1} \;(1-q)^4\prod_{n=1}^\infty\dfrac{(1-q^{2n})^8}{(1-q^{2n-1})^8}=\dfrac{\pi^4}{16}
%\end{eqnarray} 
%and,
\begin{eqnarray}
\lim_{q\uparrow 1} \;(1-q)^{2k}\prod_{n=1}^\infty\dfrac{(1-q^{2n})^{4k}}{(1-q^{2n-1})^{4k}}=\dfrac{\pi^{2k}}{2^{2k}}
\end{eqnarray} 
where $q\uparrow 1$ indicates $q\to 1$ from within the unit disk. We treat Theorems 3.1 and 3.2 separately. 
\subsection{Proof of Theorem 3.1}
Let $k\geq 2$ be an even integer. Then from $(\ref{Eis})$ and $(\ref{ModD})$ we have
\begin{eqnarray*}
H_{2k}(\tau)=\sum_{n=1}^\infty \sigma^{\#}_{2k-1}(2n)q^{2n}
= 2^{2k-1}\sum_{n=1}^\infty \sigma^{\#}_{2k-1}(n)q^{2n}.
\end{eqnarray*}
Using the definition of $\sigma^{\#}_{2k-1}(n)$ in the expression above we obtain
\begin{eqnarray}\label{brac}
H_{2k}(\tau) &=& 2^{2k-1}\sum_{n=1}^\infty\left(\sum_{\substack{d|n\\n/d\;odd}}d^{2k-1}\right)q^{2n}\nonumber\\
&=& 2^{2k-1}\sum_{i=0}^\infty\sum_{j=0}^\infty (j+1)^{2k-1}q^{2(j+1)(2i+1)}\nonumber\\
&=& 2^{2k-1}\sum_{i=0}^\infty\left(\sum_{j=0}^\infty (j+1)^{2k-1}(q^{2(2i+1)})^{j+1}\right).
\end{eqnarray}
We wish to find a polynomial $Q^e_{2k-1}(z)$ such that the expression in parentheses in the right-hand equation of $(\ref{brac})$ can be written as
\begin{eqnarray}\label{poly}
\dfrac{Q^e_{2k-1}(q^{2(2i+1)})}{(1-q^{2(2i+1)})^{2k}}=\sum_{j=0}^\infty (j+1)^{2k-1}(q^{2(2i+1)})^{j+1}.
\end{eqnarray}  
For notational simplicity let us write $z=q^{2(2i+1)}$ so that $(\ref{poly})$ can be rewritten as
\begin{eqnarray}\label{polysim}
\dfrac{Q^e_{2k-1}(z)}{(1-z)^{2k}}=\sum_{j=0}^\infty (j+1)^{2k-1}z^{j+1}.
\end{eqnarray}
\begin{lem}\label{polyeis}
$Q^e_{2k-1}(z)$ is a polynomial of degree $(2k-1)$ with integer coefficients.
\end{lem}
\begin{proof}
The right hand side of $(\ref{polysim})$ can be identified with the left-hand side of (\ref{genS}) so that using Proposition \ref{Stir} we can rewrite the right-hand side of (\ref{polysim}) as
\begin{eqnarray}
\dfrac{Q^e_{2k-1}(z)}{(1-z)^{2k}}= \sum_{j=0}^{2k-1}j!\stirling{2k-1}{j}\dfrac{z^j}{(1-z)^{j+1}}
\end{eqnarray}
Noting that $z=1-(1-z)$ we use binomial expansion in $z^j=(1-(1-z))^j$ followed by rearrangements of the sums above to obtain
\begin{eqnarray}\label{ak}
\dfrac{Q^e_{2k-1}(z)}{(1-z)^{2k}}&=& \sum_{j=0}^{2k-1}j!\stirling{2k-1}{j}\dfrac{(1-(1-z))^j}{(1-z)^{j+1}}\nonumber\\
&=& \sum_{j=0}^{2k-1}j!\stirling{2k-1}{j}\sum_{m=0}^j(-1)^{j-m}\binom{j}{m}(1-z)^{j-m}\dfrac{1}{(1-z)^{j+1}}\nonumber\\
&=& \sum_{m=0}^{2k-1}(-1)^{m}\left(\sum_{j=0}^{2k-1} (-1)^j j!\stirling{2k-1}{j}\binom{j}{m}\right)\dfrac{1}{(1-z)^{m+1}}\nonumber\\
&=& \sum_{m=0}^{2k-1}\dfrac{(-1)^m a_{k}(m)}{(1-z)^{m+1}},
\end{eqnarray}
where $a_{k}(m)$ is defined by
\begin{eqnarray}
a_k(m) = \sum_{j=0}^{2k-1} (-1)^j j!\stirling{2k-1}{j}\binom{j}{m}.
\end{eqnarray}
Note that in going from the second step to the third step in (\ref{ak}) we used the fact that $\binom{j}{m}=0$ if $m>j$ and hence we are able to interchange the sums over $m$ and $j$ above. Thus multiplying both sides of (\ref{ak}) by $(1-z)^{2k}$ and using the binomial expansion yields 
\begin{eqnarray}
Q^e_{2k-1}(z)&=&\sum_{m=0}^{2k-1}(-1)^m a_{k}(m)(1-z)^{2k-m-1}\nonumber\\
&=& \sum_{m=0}^{2k-1}(-1)^m a_{k}(m)\sum_{l=0}^{2k-m-1}(-1)^l\binom{2k-m-1}{l}z^l\nonumber\\
&=& \sum_{l=0}^{2k-1}(-1)^l\left(\sum_{m=0}^{2k-1}(-1)^m a_{k}(m)\binom{2k-m-1}{l}\right)z^l\nonumber\\
&=& \sum_{l=0}^{2k-1}(-1)^l b_k(l)z^l,
\end{eqnarray}
where the $b_k(l)$ are defined by
\begin{eqnarray}
b_k(l) = \sum_{m=0}^{2k-1}(-1)^m a_{k}(m)\binom{2k-m-1}{l},
\end{eqnarray}
which establishes Lemma \ref{polyeis}.
\end{proof}
We also note from (\ref{polysim}) that $Q^e_{2k-1}(0)=0$. Thus we define the polynomial $P^e_{2k-2}(z)$ of degree $(2k-2)$ by 
\begin{eqnarray}\label{newpoly}
Q^e_{2k-1}(z):=zP^e_{2k-2}(z)
\end{eqnarray}
where 
\begin{eqnarray}
P^e_{2k-2}(z)=\sum_{l=1}^{2k-1}(-1)^l b_k(l)z^{l-1}.
\end{eqnarray} 
We rewrite (\ref{brac}) using (\ref{newpoly}) as
\begin{eqnarray}\label{eispolynew}
H_{2k}(\tau)=2^{2k-1}\sum_{i=0}^\infty \dfrac{q^{2(2i+1)}P^e_{2k-2}(q^{2(2i+1)})}{(1-q^{2(2i+1)})^{2k}}.
\end{eqnarray}
Now from (\ref{psi}) of Theorem \ref{Ono} we have
\begin{eqnarray}\label{Onore}
H_{2k}(\tau) - T_{2k}(\tau) = d_kq^k\psi^{4k}(q^2).
\end{eqnarray} 
Thus from (\ref{Gausscor}), (\ref{eispolynew}) and (\ref{Onore}) we get
\begin{eqnarray}\label{resu}
\ \ \ \ \ \ \ \ \ \ \ \  \sum_{n=0}^\infty \dfrac{2^{2k-1}q^{2(2n+1)}P^e_{2k-2}(q^{2(2n+1)})}{(1-q^{2(2n+1)})^{2k}}-T_{2k}(\tau)=d_kq^k\prod_{n=1}^\infty\dfrac{(1-q^{4n})^{2k}}{(1-q^{4n-2})^{2k}}.
\end{eqnarray}
Making the change of variable $q\rightarrow \sqrt{q}$ in (\ref{resu}) we obtain
\begin{eqnarray}\label{modiresu}
\ \ \ \ \ \ \ \ \ \ \ \sum_{n=0}^\infty \dfrac{2^{2k-1}q^{2n+1}P^e_{2k-2}(q^{2n+1})}{(1-q^{2n+1})^{2k}}-T_{2k}(\tau/2)=d_kq^{k/2}\prod_{n=1}^\infty\dfrac{(1-q^{2n})^{4k}}{(1-q^{2n-1})^{4k}}.
\end{eqnarray}
On multiplying both sides of (5.7) by $(1-z)^{2k}$ we obtain 
\begin{eqnarray}\label{sumvan}
Q^e_{2k-1}(z)=\sum_{j=0}^{2k-1}j!\stirling{2k-1}{j}z^j(1-z)^{2k-j-1}.
\end{eqnarray}
As $z\rightarrow 1^-$, each summand in (\ref{sumvan}) vanishes except the term corresponding to $j=2k-1$. Thus we get 
\begin{eqnarray}\label{limit1}
\lim_{z\rightarrow 1^-} Q^e_{2k-1}(z)= (2k-1)!.
\end{eqnarray}
In view of (\ref{limit1}) and the fact that $T_{2k}(\tau/2)\rightarrow 0$ (cusp form), as $q\rightarrow 1^-$, (\ref{modiresu}) gives 
\begin{eqnarray}
2^{2k-1}\sum_{n=0}^\infty \dfrac{(2k-1)!}{(2n+1)^{2k}}=\dfrac{d_k\pi^{2k}}{2^{2k}}.
\end{eqnarray}
Using the definition of $d_k$ we obtain identity (5.1). Thus Theorem \ref{main1} follows from all the above observations. 
\subsection{Proof of Theorem 3.2}
Let $k\geq 1$ be an odd integer. Then from (\ref{ModD}) and (\ref{Eis}) we have
\begin{eqnarray}\label{oddpoly}
H_{2k}(\tau) &=& \sum_{\substack{n>0\\n\;odd}}\sigma_{2k-1}^{\#}(n)q^n\nonumber\\
&=&\sum_{\substack{n=0\\n\;odd}}^\infty\sigma_{2k-1}^{\#}(2n+1)q^{2n+1}\nonumber\\
&=&\sum_{n=1}^\infty\sigma_{2k-1}^{\#}(n)q^n-\sum_{n=1}^\infty\sigma_{2k-1}^{\#}(2n)q^{2n}\nonumber\\
&=&\sum_{n=1}^\infty\left(\sum_{\substack{d|n\\n/d\;odd}}d^{2k-1}\right)q^n-2^{2k-1}\sum_{n=1}^\infty\left(\sum_{\substack{d|n\\n/d\;odd}}d^{2k-1}\right)q^{2n}\nonumber\\
&=& \sum_{i=0}^\infty\sum_{j=0}^\infty\left(j+1\right)^{2k-1}q^{(j+1)(2i+1)}-2^{2k-1}\sum_{i=0}^\infty\sum_{j=0}^\infty (j+1)^{2k-1}q^{2(j+1)(2i+1)}.
\end{eqnarray}
In view of (\ref{polysim}) and Lemma \ref{polyeis} we can rewrite (\ref{oddpoly}) as
\begin{eqnarray}\label{soddpoly}
H_{2k}(\tau) &=& \sum_{i=0}^\infty \dfrac{Q^e_{2k-1}(q^{2i+1})}{(1-q^{2i+1})^{2k}}-2^{2k-1}\sum_{i=0}^\infty \dfrac{Q^e_{2k-1}(q^{2(2i+1)})}{(1-q^{2(2i+1)})^{2k}}\nonumber\\
&=&\sum_{i=0}^\infty \dfrac{Q^o_{4k-2}(q^{2i+1})}{(1-q^{2(2i+1)})^{2k}}
\end{eqnarray}
where $Q^o_{4k-1}(q^{2i+1})$ is the polynomial in $w=q^{2i+1}$ of degree $4k-1$ defined by
\begin{eqnarray}\label{oddrecpoly}
Q^o_{4k-1}(w)=(1+w)^{2k}Q^e_{2k-1}(w)-2^{2k-1}Q^e_{2k-1}(w^2).
\end{eqnarray} 
Since $Q^e_{2k-1}(0)=0$, in view of (\ref{polysim}) and (\ref{oddrecpoly}) we also have $Q^o_{4k-1}(0)=0$. Therefore we define the polynomial $P^o_{4k-2}(w)$ of degree $4k-2$ by 
\begin{eqnarray}\label{podd}
Q^o_{4k-1}(w)&=& w(1+w)^{2k}P^e_{2k-2}(w)-2^{2k-1}w^2P^e_{2k-2}(w^2)\nonumber\\
&:=& wP^o_{4k-2}(w)
\end{eqnarray} 
where
\begin{eqnarray}
P^o_{4k-2}(w)=(1+w)^{2k}P^e_{4k-2}(w)-2^{2k-1}wP^e_{4k-2}(w^2).
\end{eqnarray}
Hence from (\ref{soddpoly}), (\ref{podd}) and Theorem \ref{Ono} we obtain
\begin{eqnarray}\label{finalodd}
&& H_{2k}(\tau) - T_{2k}(\tau) = d_kq^k\psi^{4k}(q^2),\nonumber\\
&& \sum_{n=0}^\infty \dfrac{q^{2n+1}P^o_{4k-2}(q^{2n+1})}{(1-q^{2(2n+1)})^{2k}}-T_{2k}(\tau)=d_kq^k\prod_{n=1}^\infty\dfrac{(1-q^{4n})^{2k}}{(1-q^{4n-2})^{2k}}.
\end{eqnarray}
Also as $w\rightarrow 1^-$, (\ref{limit1}) and (\ref{podd}) yield 
\begin{eqnarray}\label{limit2}
\ \ \ \ \ \lim_{w\rightarrow 1^-}Q^o_{4k-1}(w)=2^{2k}(2k-1)!-2^{2k-1}(2k-1)!=2^{2k-1}(2k-1)!.
\end{eqnarray}
Thus on multiplying both sides of (\ref{finalodd}) by $(1-q^2)^{2k}$ and taking the  limit as $q\rightarrow 1$ from within the unit disk, we obtain the following using (\ref{limit2}):
\begin{eqnarray}
2^{2k-1}(2k-1)!\sum_{n=0}^\infty\dfrac{1}{(2n+1)^{2k}}=\dfrac{d_k\pi^{2k}}{2^{2k}},
\end{eqnarray}
Using definition of $d_k$ we immediately obtain identity (5.1). Thus Theorem \ref{main2} follows from all of the above observations. 
\section{Explicit computations of $P^e_{2k-2}(z)$ and $P^o_{4k-2}(z)$ for different $k$}
We used the Python programming language to compute the co-efficients $a_k(m)$ and $b_k(l)$ to determine the polynomials $P^e_{2k-2}(z)$ and $P^o_{4k-2}(z)$ for a few different values of $k$. We will see that our results for $k=1, 2, 3$ tally with the results in \cite{Sun} and \cite{AG}.
\subsection{Case k=1 : Sun's result}
Since $k=1$ is odd we use (\ref{polyodd}) to get 
\begin{eqnarray*}
P^o_{2}(z)=(1+z)^2P^e_{0}(z)-2zP^e_{0}(z),
\end{eqnarray*}
where we define $P^e_{0}(z)=1$. Therefore, 
\begin{eqnarray}\label{case1} 
P^o_{2}(z)=(1+z)^2-2z=1+z^2.
\end{eqnarray}
Thus, (\ref{case1}) and (\ref{qaz5}) yield
\begin{eqnarray}\label{qana1}
\sum_{n=1}^\infty \dfrac{q^{2n}(1+q^{2(2n+1)})}{(1-q^{2(2n+1)})^{2}}=\prod_{n=1}^\infty\dfrac{(1-q^{4n})^{4}}{(1-q^{4n-2})^{4}},
\end{eqnarray}
where, $d_1=1$ and $T_{2}(\tau)=0$, (Table I, p.120, \cite{ARBILEO}). This is Theorem 1.1, (1.1), of \cite{Sun} with $q\rightarrow\sqrt{q}$ in (\ref{qana1}).
\subsection{Case k=2}
Here $k=2$ is even, so we use (\ref{polyeven}) to get
\begin{eqnarray*}
P^e_{2}(z)=-b_2(1)+b_2(2)z-b_2(3)z^2
\end{eqnarray*}
where $b_2(l)$, $1\leq l\leq 3$ are given by (\ref{evenbs}). We need to evaluate $a_2(0), a_2(1), a_2(2)$\\ and $a_2(3)$. Using (\ref{evenas}) we obtain 
\begin{eqnarray*}
a_2(0)=-1,\;a_2(1)=-7,\;a_2(2)=-12,\;a_2(3)=-6.
\end{eqnarray*}
Thus (\ref{evenbs}) yields
\begin{eqnarray*}
b_2(1)=-1,\;b_2(2)=4,\;b_2(3)=-1.
\end{eqnarray*}
Hence we obtain
\begin{eqnarray}
P^e_{2}(z)=1+4z+z^2
\end{eqnarray}
which when used in (\ref{qaz4}) yields the results in \cite{Sun} and \cite{AG}. Here again $T_{4}(\tau/2)=0$ (Table I, p.120, \cite{ARBILEO}).
\subsection{Case: k=3}
Here we need to evaluate the coefficients $b_3(1), b_3(2), b_3(3), b_3(4),$\\$b_3(5)$ and the corresponding $a_3(0), a_3(1), a_3(2), a_3(3), a_3(4), a_3(5)$. Using (\ref{evenas}) we get
\begin{eqnarray*}
&&a_3(0)=-1,\;a_3(1)=-31,\;a_3(2)=-180,\;a_3(3)=-390,\;a_3(4)=-360,\\&&a_3(5)=-120
\end{eqnarray*}
and using (\ref{evenbs}) we obtain
\begin{eqnarray*}
b_3(1)=-1,\;b_3(2)=26,\;b_3(3)=-66,\;b_3(4)=26,\;b_3(5)=-1.
\end{eqnarray*}
Using these values in (\ref{polyodd}) we obtain
\begin{eqnarray*}
P^o_{10}(z)&=&(1+z)^6P^e_4(z)-32zP^e_4(z^2)\\
&=&(1+z)^6(1+26z+66z^2+26z^3+z^4)\\&&-32z(1+26z^2+66z^4+26z^6+z^8)\\
&=& z^{10}+237z^8+1682z^6+1682z^4+237z^2+1\\
&=&(z^2+1)(z^8+236z^6+1446z^4+236z^2+1)
\end{eqnarray*}
which when used in (\ref{qaz5}) along with the change of variable $q\rightarrow \sqrt{q}$ gives us the result in \cite{AG}. We note here that in \cite{AG} we obtained explicitly $T_{6}(\tau/2)=\phi^{12}(q)$ where $\phi(q)=\prod_{n=1}^\infty (1-q^n)$ is the Euler's function.
\subsection{Case: k=4}
Here we use (\ref{evenas}) and (\ref{evenbs}) to obtain
\begin{eqnarray*}
&& a_4(0)=-1,\; a_4(1)=-127,\;a_4(2)=-1932,\;a_4(3)=-10206,\;a_4(4)=-25200,\\&& a_4(5)=-31920,\;a_4(6)=-20160,\;a_4(7)=-5040
\end{eqnarray*} 
and 
\begin{eqnarray*}
&&b_4(1)=-1,\;b_4(2)=120,\;b_4(3)=-1191,\;b_4(4)=2416,\;b_4(5)=-1191,\\&& b_4(6)=120, b_4(7)=-1.
\end{eqnarray*}
Thus we have
\begin{eqnarray*}
P^e_6(z)=z^6+120z^5+1191z^4+2416z^3+1191z^2+120z+1
\end{eqnarray*}
which when used in (\ref{qaz4}) gives us the following $q$-analogue of $\zeta(8)=\pi^8/9450$
\begin{eqnarray}
\sum_{n=0}^\infty \dfrac{q^{2n}\;P^e_{6}(q^{2n+1})}{(1-q^{2n+1})^{8}}-T_{8}(\tau/2)=136q\prod_{n=1}^\infty\dfrac{(1-q^{2n})^{16}}{(1-q^{2n-1})^{16}}.
\end{eqnarray}
\subsection{Case: k=5}
Here we use (\ref{evenas}) and (\ref{evenbs}) to obtain
\begin{eqnarray*}
&& a_5(0)=-1,\; a_5(1)=-511,\;a_5(2)=-18660,\;a_5(3)=-204630,\\&&a_5(4)=-1020600, a_5(5)=-2739240,\;a_5(6)=-4233600,\;a_5(7)=-3780000,\\&&a_5(8)=-1814400,a_5(9)=-362880
\end{eqnarray*} 
and 
\begin{eqnarray*}
&&b_5(1)=-1,\;b_5(2)=502,\;b_5(3)=-14608,\;b_5(4)=88234,b_5(5)=-156190,\\&&b_5(6)=88234, b_5(7)=-14608,\;b_5(8)=502,\;b_5(9)=-1.
\end{eqnarray*}
Thus from (\ref{polyodd}) we obtain the polynomial 
\begin{eqnarray*}
P^o_{18}(z)&=&(1+z)^{10}P^e_{8}(z)-1024zP^e_{8}(z^2)\\
&=& (1+z)^{10}\left(z^8+502z^7+14608z^6+88234z^5+156190z^4+88234z^3\right.\\ &&\left.+14608z^2+502z+1\right) -512z\left(z^{16}+502z^{14}+14608z^{12}+88234z^{10}\right.\\&&\left.+156190z^8+88234z^6+14608z^4+502z^2+1\right)\\
&=& (1+z^2)(z^{16}+19672z^{14}+1736668z^{12}+19971304z^{10}+49441990z^8\\&&+19971304z^6+1736668z^4+19672z^2+1).
\end{eqnarray*}
Using this in (\ref{qaz5}) with $q\rightarrow\sqrt{q}$ we obtain the following $q$-analogue of $\zeta(10)=\pi^{10}/93555$ 
\begin{eqnarray*}
\sum_{n=0}^\infty \dfrac{q^{n}(1+q^{2n+1})\;S_{8}(q^{2n+1})}{(1-q^{2n+1})^{10}}-T_{10}(\tau/2)=2031616q^2\prod_{n=1}^\infty\dfrac{(1-q^{2n})^{20}}{(1-q^{2n-1})^{20}}
\end{eqnarray*}
where 
\begin{eqnarray*}
S_8(z)&=&z^{8}+19672z^{7}+1736668z^{6}+19971304z^{5}+49441990z^4\\&&+19971304z^3+1736668z^2+19672z+1.
\end{eqnarray*}
\section{Acknowledgement}
The author is grateful to Krishnaswami Alladi for his constant support, encouragement and stimulating discussions. He sincerely thanks Frank Garvan for several interesting discussions on the problem and providing him with some useful references. He also expresses his appreciation to George Andrews for his support. Finally, he thanks the anonymous referees for their feedback on the manuscript which improved exposition. 

\end{document}